\def\PREP{PREP}
\def\PUBL{PUBL}
\def\form{PREP}

\ifx\form\PREP
\documentclass[12pt]{amsart}
\fi
\ifx\form\PUBL
\documentclass{sig-alternate}
\fi

\usepackage{amssymb, colordvi, multirow, graphicx, algorithmic, setspace, color}

\def \<{\langle}
\def \>{\rangle}
\def\BS{Bernstein--Sato}
\def\Ann{\operatorname{Ann}}
\def\NF{\operatorname{NF}}
\def\locusB{\operatorname{exceptionalLocusB}}
\def\locusCore{\operatorname{exceptionalLocusCore}}
\def\localB{\operatorname{localBFunction}}
\def\generalB{\operatorname{generalB}}
\def\multiplierIdeal{\operatorname{multiplierIdeal}}
\def\multiplierIdealLinAlg{\operatorname{multiplierIdealLA}}
\def\starIdeal{\operatorname{starIdeal}}
\def\LAtrick{\operatorname{linearAlgebraTrick}}

\def\gr{\operatorname{gr}}
\def\Gr{\operatorname{Gr}}
\def\Spec{\operatorname{Spec}}

\newcommand{\rr}{{\mathbb R}}
\newcommand{\cc}{{\mathbb C}}
\newcommand{\bC}{{\mathbb C}}
\newcommand{\bN}{{\mathbb N}}
\newcommand{\bZ}{{\mathbb Z}}
\newcommand{\bQ}{{\mathbb Q}}

\newcommand{\bP}{{\mathbb P}}

\newcommand{\bff}{{\boldsymbol{f}}}
\newcommand{\bfs}{{\boldsymbol{s}}}
\newcommand{\bft}{{\boldsymbol{t}}}
\newcommand{\bffs}{{\bff^\bfs}}

\newcommand{\bfp}{{\boldsymbol{\partial}}}

\newcommand{\bfx}{{\boldsymbol{x}}}

\newcommand{\zero}{{\mathbf{0}}}

\newcommand{\p}{{\partial}}

\newcommand{\IN}{{\operatorname{in}}}

 \def\D0{D_\zero}

\def\ol#1{{\overline {#1}}}
\def\lct{\operatorname{lct}}
\newcommand{\multI}[2]{{\mathcal J}({#1}^{#2})}
\newcommand{\Dx}{{D_X}}
\newcommand{\Dxt}{{D_Y}}

\catcode`\@=11\relax \newdimen\p@renwd \font\tenex=cmex10
\setbox0=\hbox{\tenex B} \p@renwd=\wd0
\def\bbordermatrix#1{\begingroup \m@th
\setbox\z@\vbox{\def\\{\crcr\noalign{\kern2\p@\global\let\cr\endline}}%
    \ialign{$##$\hfil\kern2\p@\kern\p@renwd&\thinspace\hfil$##$\hfil
      &&\quad\hfil$##$\hfil\crcr
      \omit\strut\hfil\crcr\noalign{\kern-\baselineskip}%
      #1\crcr\omit\strut\cr}}%
  \setbox\tw@\vbox{\unvcopy\z@\global\setbox\@ne\lastbox}%
  \setbox\tw@\hbox{\unhbox\@ne\unskip\global\setbox\@ne\lastbox}%
  \setbox\tw@\hbox{$\kern\wd\@ne\kern-\p@renwd\left[\kern-\wd\@ne
    \global\setbox\@ne\vbox{\box\@ne\kern2\p@}%
    \vcenter{\kern-\ht\@ne\unvbox\z@\kern-\baselineskip}\,\right]$}%
  \null\;\vbox{\kern\ht\@ne\box\tw@}\endgroup}
\catcode`\@=12\relax

\ifx\form\PREP
\fi
\numberwithin{equation}{section}
\newtheorem{theorem}{Theorem}
\numberwithin{theorem}{section}
\newtheorem{proposition}[theorem]{Proposition}
\newtheorem{lemma}[theorem]{Lemma}
\newtheorem{cor}[theorem]{Corollary}

\newtheorem{ex}[theorem]{Example}
\newtheorem{rem}[theorem]{Remark}
\newtheorem{definition}[theorem]{Definition}
\newtheorem{algorithm}[theorem]{Algorithm}

\newenvironment{example}{\begin{ex}\rm}{\end{ex}}
\newenvironment{remark}{\begin{rem}\rm}{\end{rem}}

\begin{document}
\ifx\form\PREP
\def\publname{\scriptsize \Red{Draft of \today} \def\currentvolume{}
\def\currentissue{} \pagespan{1}{60} \PII{}} \copyrightinfo{}{}
\fi
\ifx\form\PREP
\title
[Algorithms for Bernstein--Sato polynomials and multiplier ideals]
{Algorithms for Bernstein--Sato polynomials\\ and multiplier ideals}
\author{Christine Berkesch}
\email{cberkesc@math.purdue.edu}
\author{Anton Leykin}
\email{leykin@math.gatech.edu}
\fi
\ifx\form\PUBL
\title{Algorithms for Bernstein--Sato polynomials\\ and multiplier ideals}
\numberofauthors{2}
\author{
\alignauthor
Christine Berkesch\\
       \affaddr{Department of Mathematics}\\
       \affaddr{Purdue University}\\
       \email{cberkesc@purdue.edu}\\
\alignauthor
Anton Leykin\\
       \affaddr{School of Mathematics}\\
       \affaddr{Georgia Institute of Technology}\\
       \email{leykin@math.gatech.edu}
}
  \conferenceinfo{ISSAC 2010,}{25--28 July 2010, Munich, Germany.}
  \CopyrightYear{2010}
  \crdata{978-1-4503-0150-3/10/0007}
\fi
\ifx\form\PUBL
\maketitle
\fi
\begin{abstract}
The Bernstein--Sato polynomial (or global $b$-function) is an important invariant in singularity theory, which can be computed using symbolic methods in the theory of $D$-modules.
After providing a survey of known algorithms for computing the global $b$-function, we develop a new method to compute the local $b$-function for a single polynomial.
We then develop algorithms that compute generalized Bernstein--Sato polynomials of Budur--Musta\c{t}\v{a}--Saito and Shibuta for an arbitrary polynomial ideal.
These lead to computations of log canonical thresholds, jumping coefficients, and multiplier ideals.
Our algorithm for multiplier ideals simplifies that of Shibuta and shares a common subroutine with our local $b$-function algorithm.
The algorithms we present have been implemented in the $D$-modules package of the computer algebra system Macaulay2.
\end{abstract}
\ifx\form\PREP
\maketitle
\fi

\ifx\form\PUBL


\fi

\section{I\lowercase{ntroduction}}
\label{sec:intro}

The multiplier ideals of an algebraic variety carry essential information about its singularities and have proven themselves a powerful tool in algebraic geometry.
However, they are notoriously difficult to compute;
nice descriptions are known only for very special families of varieties,
such as monomial ideals and hyperplane arrangements
\cite{mult monomial, mircea hyp mult, zach hyp mult, saito on}.
To briefly recall the definition of this invariant, let $X = \cc^n$
with coordinates $\bfx = x_1,\dots, x_n$.
For an ideal $\<\bff\> = \<f_1,\dots,f_r\> \subseteq \cc[\bfx]$
and a nonnegative rational number $c$,
the \emph{multiplier ideal of $\bff$ with coefficient $c$} is 
\begin{align*}
\multI{\bff}{c} =
\left\{ h\in \cc[\bfx] \ \bigg\vert\ \frac{|h|^2}{(\sum |f_i|^2)^c} \text{ is locally integrable} \right\}.
\end{align*}
It follows from this definition that $\multI{\bff}{c} \supseteq \multI{\bff}{d}$
for $c\leq d$ and $\multI{\bff}{0} = \cc[\bfx]$ is trivial.
The (global) \emph{jumping coefficients} of $\bff$ are a discrete sequence of
rational numbers $\xi_i = \xi_i(\bff)$ with
$0 = \xi_0 < \xi_1 < \xi_2 < \cdots$
satisfying the property that $\multI{\bff}{c}$ is constant exactly for $c\in[\xi_i,\xi_{i+1})$.
In particular, the \emph{log canonical threshold} of $\bff$ is $\xi_1$,
denoted by $\lct(\bff)$.
This is the least rational number $c$ for which $\multI{\bff}{c}$ is nontrivial.
The multiplier ideal $\multI{\bff}{c}$ measures the singularities of
the variety of $\bff$ in $X$; smaller multiplier ideals
(and lower log canonical threshold) correspond to worse singularities.
For an equivalent algebro-geometric definition and an introduction to this invariant,
we refer the reader to \cite{posII,short course}.

In this paper we develop an algorithm for computing multiplier ideals and jumping coefficients by way of an even finer invariant, \BS\ polynomials, or $b$-functions.
The results of Budur~{\em et al.}~\cite{bms} provide other applications
for our \BS\ algorithms, including multiplier ideal membership tests,
an algorithm to compute jumping coefficients,
and a test to determine if a complete intersection has at most rational singularities.

The first $b$-function we consider, the \emph{global \BS\ polynomial} of a
hypersurface, was introduced independently by Bernstein~\cite{Bernstein}
and Sato~\cite{Sato:B-S-polynomial}.
This univariate polynomial plays a central role in the theory of $D$-modules
(or algebraic analysis), which was founded by, amongst others, Kashiwara~
\cite{Kashiwara} and Malgrange~\cite{Malgrange}.
Moreover, the jumping coefficients of $\bff$ that lie in the interval $(0,1]$
are roots of its global \BS\ polynomial \cite{elsv};
however, this $b$-function contains more information.
Its roots need not be jumping coefficients, even if they are between $0$ and $1$
(see Example~\ref{ex:nonJC root}).

The \BS\ polynomial was recently generalized by
Budur~{\em et al.}~\cite{bms} to arbitrary varieties.
The maximal root of this \emph{generalized \BS\ polynomial} provides
a multiplier ideal membership test.
Shibuta defined another generalization to compute explicit generating sets for multiplier ideals \cite{shibuta}.
Our multiplier ideal algorithm employs the $b$-functions os Shibuta,
which we call the \emph{$m$-generalized \BS\ polynomial}.
However, it circumvents primary decomposition
and one elimination step through a syzygetic technique
(see Algorithms~\ref{alg: LAtrick} and \ref{alg: exceptional locus core}).
The correctness of our results relies heavily on the use of
$V$-filtra\-tions, as developed by Kashiwara and Malgrange \cite{Kashiwara:V-filtration, Malgrange:V-filtration}.

$D$-module computations are made possible by Gr\"obner bases techniques in the Weyl algebra.
The computation of the \BS\ polynomial was pioneered by Oaku in \cite{Oaku:local-b}.
His algorithm was one of the first algorithms in algebraic analysis, many of which are outlined in the book by Saito~{\em et al.}~\cite{HypergeomBook}.
The computation of the \emph{local \BS\ polynomial}
was first addressed in the early work of Oaku~\cite{Oaku:local-b},
as well as the recent work of Nakayama~\cite{Nakayama:local-b-function},
Nishiyama and Noro~\cite{Nishiyama-Noro:stratification-by-local-b},
and Schulze \cite{mathias differential, mathias normal}.
Bahloul and Oaku~\cite{Bahloul-Oaku:local-bs-ideals} address the computation of local \BS\ ideals that generalize \BS\ polynomials.
In this article we provide our version of the local algorithm for \BS\ polynomials,
part of which is vital to our approach to computation of multiplier ideals.

There are several implementations of algorithms for global and local $b$-functions in kan/sm1~\cite{KANwww}, Risa/Asir~\cite{risa-asir-www}, and Singular~\cite{GPS01}.
One can find a comparison of performance in
\cite{Levandovskyy-Morales:comparison}.
All of the algorithms in this article have been implemented
and can be found in the $D$-modules package~\cite{DmodulesM2}
of the computer algebra system Macaulay2~\cite{M2}.

\smallskip

The first author was partially supported by
NSF Grants DMS 0555319 and DMS 090112; the second author is partially supported by the NSF Grant DMS 0914802.

\subsection*{Outline}
\label{subsec:outline}

Section~\ref{sec:global} surveys the known approaches for computing
the global \BS\ polynomial, highlighting an algorithm of Noro \cite{NoroBPoly}.
In Section~\ref{sec:local}, we present an algorithm for computing the local \BS\ polynomial.
Algorithms for the generalized \BS\ polynomial for an arbitrary variety, as introduced by Budur~{\em et al.}~\cite{bms}, are discussed in Section~\ref{sec:generalized}, along with their applications.
Based on the methods of Section~\ref{sec:local},
Section~\ref{sec:bsm+mult} considers the $m$-generalized \BS\ polynomial of Shibuta \cite{shibuta} and contains our algorithms for multiplier ideals.

\section{G\lowercase{lobal} \BS\ \lowercase{polynomials}}
\label{sec:global}

Let $K$ be a field of characteristic zero, and set $X = K^n$ and $Y = X\times K$ with coordinates $(\bfx)$ and $(\bfx,t)$, respectively.
We consider the $n$-th Weyl algebra $\Dx = K\<\bfx,\bfp\>$ with generators $x_1,\ldots, x_n$ and $\p_{x_1},\ldots, \p_{x_n}$, as well as $\Dxt = K\<\bfx,\bfp_\bfx,t,\p_t\>$, the Weyl algebra on $Y$.
Define an action of $\Dxt$ on $N_f := K[\bfx][f^{-1},s]f^s$ as follows: $x_i$ and $\p_{x_i}$ act naturally for $i = 1, \ldots, n$, and
\begin{align*}
  t \cdot h(\bfx, s) f^s = h(\bfx, s+1) f f^s \quad \text{and}\quad 
  \p_t \cdot h(\bfx, s) f^s = -s h(\bfx, s-1) f^{-1} f^s,
\end{align*}
where $h \in K[\bfx][f^{-1},s]$.

Let $\sigma = -\p_t t$. For a polynomial $f \in K[\bfx]$, the \emph{global \BS\ polynomial} of $f$, denoted $b_f$, is the monic polynomial $b(s)\in K[s]$ of minimal degree satisfying the equation
\begin{align}\label{eq:global}
b(\sigma) f^s = P f f^s
\end{align}
for some $P\in \Dx \< \sigma \>$.

There is an alternate definition for the global \BS\ polynomial in terms of $V$-filtrations.
To provide this, we denote by $V^\bullet \Dxt$ the $V$-filtration of $\Dxt$ along $X$, where $V^m\Dxt$ is $\Dx$-generated by the set
$\{ t^\mu \p_t^\nu \mid \mu - \nu \geq m \}$.
Let $i_f: X \rightarrow Y$ defined by $i_f(\bfx) = (\bfx,f(\bfx))$ be the graph of $f$.
The $D$-module direct image of $K[\bfx]$ along $i_f$ is the module
\[
M_f := (i_f)_+ K[\bfx] \cong K[\bfx]\otimes_K K\<\p_t\>
\]
with actions of a vector field $\xi$ on $X$ and $t$,
\begin{align*}
\xi(p\otimes \p_t^\nu) = \xi p \otimes \p_t^\nu - (\xi f) p\otimes \p_t^{\nu + 1} 
\quad\text{and}\quad 
t\cdot (p\otimes \p_t^\nu) =  f p \otimes \p_t^\nu - \nu p \otimes \p_t^{\nu - 1},
\end{align*}
providing a $\Dxt$-module structure.
Notice that there is a canonical embed ding of $M_f$ into $N_f$, where $s$ is identified with $-\p_t t$.

With $\delta = 1\otimes 1\in M_f$, the global \BS\ polynomial $b_f$ is equal to the minimal polynomial of the action of $\sigma$ on the module $(V^0 \Dxt)\delta / (V^1\Dxt)\delta$.
We now survey three ways of computing this $b$-function.

\subsection{By way of an annihilator}
\label{subsec:ann}

The global \BS\ polynomial $b_f (s)$ is the minimal polynomial of $\sigma := -\p_t t$ modulo $\Ann_{\Dx[\sigma]} f^s + \Dx[\sigma] f$, where $f^s\in N_f$.
By the next result, this annihilator can be computed from the left $\Dxt$-ideal
\[
I_f = \left\< t-f, \p_1+\textstyle\frac{\p f}{\p x_1}\p_t, \ldots, \p_n+\textstyle\frac{\p f}{\p x_n}\p_t \right\>.
\]
\begin{theorem}\cite[Theorem~5.3.4]{HypergeomBook} \label{thm:one var If}
The ideal $\Ann_{D[s]} f^s$ equals the image of $I_f \cap D[\sigma]$ under the substitution $\sigma \mapsto s$.
\end{theorem}

\subsection{By way of an initial ideal}
\label{subsec:initial}

This method makes use of $w = ({\bf 0},1)\in \rr^n\times \rr$, the elimination weight vector for $X$ in $Y$.

\begin{theorem} \label{thm: BSxs}
Let $b(x,s)$ be nonzero in the polynomial ring $K[x, s]$.
Then $b(x, \sigma) \in (\IN_{(-w,w)} I_f) \cap K[x, \sigma]$ if and only if
there exists $Q\in D[s]$ satisfying the functional equation $Qf^{s+1} = b(x,s)f^s$.
In particular,
\[
\<b_f(\sigma)\> = \IN_{(-w,w)} I_f \cap K[\sigma].
\]
\end{theorem}
\begin{proof}
The action of $t$ on $N_f$ is multiplication by $f$, hence, the existence of the functional equation is equivalent to $b(x,s)\in I_f + V^1\Dxt$. The result now follows from Theorem \ref{thm:one var If}, which identifies $s$ with $\sigma$.
\end{proof}

The following algorithm provides a more economical way to compute the
global $b$-function using linear algebra.
By establishing a nontrivial $K$-linear dependency between normal forms
$\NF_G(s^i)$ with respect to a Gr\"obner basis $G$ of $\IN_{(-w,w)} I_f$,
where $0\leq i\leq d$ and $d$ is taken as small as possible,
this algorithm bypasses elimination of $\p_1,\ldots,\p_n$.
This trick was used for the first time by Noro in \cite{NoroBPoly}, where
a modular method to speed up $b$-function computations is provided as well.
We include the following algorithm for the convenience of the reader
as a similar syzygetic approach will be used in Algorithms
\ref{alg: exceptional locus core}, \ref{alg: LAtrick}, and \ref{alg: multI-lin-alg}.
Note that the coefficients of the output are, in fact, rational,
since the \emph{roots} of a $b$-function are rational \cite{Kashiwara}.

\begin{algorithm} \label{alg: globalBS} $b = globalBFunction(f,P)$
\begin{algorithmic}\ifx\form{PREP} [1] \fi
\REQUIRE a polynomial $f\in K[\bfx]$.
\ENSURE polynomial $b \in \bQ[s]$ is the \BS\ polynomial of $f$.
\STATE $G \leftarrow \mbox{Gr\"obner basis of } \IN_{(-w,w)} I_f$.
\STATE $d \leftarrow 0$.
\REPEAT
\STATE $d \leftarrow d + 1$
\UNTIL{$\exists (c_0,\ldots,c_{d})\in \bQ^{d+1}$ such that $c_d=1$ and
	\[
	\sum_{i=0}^d c_i \NF_G(s^i) = 0.
	\]}
\RETURN $\sum_{i=0}^d c_i s^i$.
\end{algorithmic}
\end{algorithm}
This approach can be exploited in a more general setting to compute
the intersection of a left ideal with a subring generated by one element
as shown in \cite{Levandovskyy-Morales:principal-intersection}.

\subsection{By way of Brian\c{c}on--Maisonobe}
\label{subsec:b-m}

This approach, which is laid out it \cite{Briancon-Maisonobe:BernsteinIdeal},
computes  the annihilator of $f^s$ in an algebra of solvable type similar to,
but different from, the Weyl algebra.
This path has been explored by Castro-Jim\'enez and Ucha~\cite{UchaCastroPBW} and implemented in Singular~\cite{GPS01} with a performance analysis given by Levandovskyy and Morales in \cite{Levandovskyy-Morales:comparison} and recent improvements outlined in \cite{Levandovskyy-Morales:principal-intersection}.

\section{L\lowercase{ocal} \BS\ \lowercase{polynomials}}
\label{sec:local}

In this section, we provide an algorithm to compute the local \BS\ polynomial of $f$ at a prime ideal of $K[\bfx]$, which is defined by replacing the use of $\Dx$ in \eqref{eq:global} by its appropriate localization.
Algorithms~\ref{alg: exceptional locus}~and~\ref{alg: exceptional locus core} use
Theorem \ref{thm: BSxs} to compute an ideal $E_b \subset K[\bfx]$ that describes
the locus of points where the $b$-function does not divide the given $b\in\bQ[s]$.

\begin{algorithm} \label{alg: exceptional locus} $E_b = \locusB(f,b)$
\begin{algorithmic}\ifx\form{PREP} [1] \fi
\REQUIRE a polynomial $f\in K[\bfx]$, a polynomial $b \in \bQ[s]$.
\ENSURE $E_b\subset K[\bfx]$ such that $\forall\  P \in \Spec K[\bfx]$,
\[
b_{f,P}\,|\,b \Leftrightarrow E_b \not\subset P.
\]
\STATE $G \leftarrow \mbox{generators of } \IN_{(-w,w)} I_f \cap K[\bfx, s]$, where $s = -\p_t t$.
\RETURN $\locusCore(G,b)$.
\end{algorithmic}
\end{algorithm}

The following subroutine computes $K[\bfx]$-syzygies between
the elements of the form $s^i g$ of $s$-degree at most $\deg b$ and $b$ itself.
It returns the projection of the syzygies onto the component corresponding to $b$.

\begin{algorithm} \label{alg: exceptional locus core} $E_b = \locusCore(f,b)$
\begin{algorithmic}\ifx\form{PREP} [1] \fi
\REQUIRE $G \subset K[\bfx,s]$, a polynomial $b \in \bQ[s]$.
\ENSURE $E_b\subset K[\bfx]$.
\STATE $G_1 \leftarrow \{\mbox{a Gr\"obner basis of } \<G\> \mbox{ w.r.t. a monomial order eliminating } s \}$.
\STATE $d \leftarrow \deg b$.
\STATE $G_2 \leftarrow \{s^i g \mid g\in G_1,\ i+\deg_s g \leq d \}$.
\STATE $S \leftarrow \ker\phi$ where
\[
\phi: K[\bfx]^{|G_2|+1}\to \bigoplus_{i=0}^d K[\bfx] s^i
\]
maps $e_i$, for $i=1,\ldots,|G_2|$, to the elements of $G_2$ and  $e_{|G_2|+1}$ to $b$. \label{line:syzygies}
\RETURN projection of $S\subset K[\bfx]^{|G_2|+1}$ onto the last coordinate.\label{line:projection}
\end{algorithmic}
\end{algorithm}
The computation of syzygies in line~\ref{line:syzygies} and projection in line~\ref{line:projection} of Algorithm~\ref{alg: exceptional locus core}
may be combined within one efficient Gr\"obner basis computation.
\begin{proof}[
\ifx\form{PREP}
Proof
\fi
of correctness of Algorithms~\ref{alg: exceptional locus}~and~\ref{alg: exceptional locus core}]
The local \BS\ polynomial $b_{f,P}$ at $P\in\Spec K[\bfx]$ divides the given
$b\in\bQ[s]$ if and only if
\[
\begin{array}{cl}
                 & Q'f^{s+1} = b_{f,P}f^s,\text{ for some }Q'\in K[\bfx]_P\otimes D[s]\\
  \Leftrightarrow    & Qf^{s+1} = hbf^s,\text{ for some }Q\in D[s],\ h\in K[\bfx]\setminus P.
\end{array}
\]
For $h\in K[\bfx]$,
\[
\begin{array}{cl}
   & Qf^{s+1} = hbf^s,\text{ for some }Q\in D[s]\\
  \Leftrightarrow & h b \in \IN_{(-w,w)} I_f \cap K[\bfx, s]\ \ \ \ \ \ \text{(by Theorem \ref{thm: BSxs})}\\
  \Leftrightarrow & h \text{ is the last coordinate of a syzygy} \\
                  & \text{in the module produced by line~\ref{line:syzygies}}\\
  \Leftrightarrow & h \in E_b.
\end{array}
\]
This proves that $b_{f,P}\,|\,b \Leftrightarrow E_b \not\subset P$.
\end{proof}

\begin{remark}[Particulars of Algorithm \ref{alg: exceptional locus}]
\label{rem:homogenize}
In order to compute generators of $\IN_{(-w,w)}I_f$, one may apply the homogenized Weyl algebra technique (for example, see~\cite[Algorithm~1.2.5]{HypergeomBook}).
Then to compute generators of $\IN_{(-w,w)} I_f \cap K[\bfx]\< t,\p_t\>$, eliminate $\bfp_\bfx$ and
apply the map $\psi$ defined as follows: for a $(-w,w)$-homogeneous $h\in K[\bfx]\<t,\p_t\>$ with $\deg_{(-w,w)} h = d$,
\[
\psi(h) = \left\{ \begin{array}{cl}t^d h, & \mbox{if }d\geq 0,\\ \p_t^{-d} h, & \mbox{if }d<0. \end{array}\right.
\]
This is the most expensive step of the algorithm.
\end{remark}

We are now prepared to compute the local \BS\ polynomial of $f$ at a prime ideal $P\subset K[\bfx]$. Its correctness follows from that of its subroutine, Algorithm~\ref{alg: exceptional locus}.

\begin{algorithm} \label{alg: localBS} $b = \localB(f,P)$
\begin{algorithmic}\ifx\form{PREP} [1] \fi
\REQUIRE a polynomial $f\in K[\bfx]$, a prime ideal $P\subset K[\bfx]$.
\ENSURE $b \in \bQ[s]$, the local Bernstein--Sato polynomial of $f$ at $P$.
\STATE $b \leftarrow  b_f$. \COMMENT{global $b$-function}
\FOR{ $r\in b_f^{-1}(0)$ }
\WHILE{ $(s-r)\,|\,b$ }
\STATE $b' \leftarrow b/(s-r)$.
\IF{$\locusB(f,b')\subset P$}
\STATE break the \emph{\bf while} loop.
\ELSE
\STATE $b\leftarrow b'$.
\ENDIF
\ENDWHILE
\ENDFOR
\RETURN b.
\end{algorithmic}
\end{algorithm}

\begin{remark}
Algorithm~\ref{alg: exceptional locus} can also be used to compute the stratification of $\Spec K[\bfx]$ according to local $b$-function. Below are the key steps in this procedure.
  \begin{enumerate}
    \item Compute the global $b$-function $b_f$.
    \item For all roots $c\in b_f^{-1}(0)$ compute
    \[
    E_{c,\,i} = \locusB(b_f/(s-c)^{\mu_c-i}),
    \]
    where $i\geq 0$ and is at most the multiplicity $\mu_c$ of the root $c$ in $b_f$.
    \item The stratum of $b = \Pi_{c\in b_f^{-1}(0)} (s-c)^{i_c}$, a divisor of $b_f$, is
    \[
    V\left(\bigcap_{c\in b_f^{-1}(0),\, i_c>0} E_{c,\,i_c-1}\right) \setminus \left(\bigcup_{c\in b_f^{-1}(0)} V(E_{c,\,i_c})\right).
    \]
  \end{enumerate}
  This approach is similar to that in the recent work
  \cite{Nishiyama-Noro:stratification-by-local-b} of Nishiyama and Noro,
  which offers a more detailed treatment.
\end{remark}

\section{G\lowercase{eneralized} \BS\ \lowercase{polynomials}}
\label{sec:generalized}

\subsection{Definitions}
For polynomials $\bff = f_1,\dots,f_r \in K[\bfx]$,
let $\bff^\bfs = \prod_{i=1}^r f_i^{s_i}$ and $Y = K^n\times K^r$
with coordinates $(\bfx,\bft)$.
Define an action of $\Dxt = K\<\bfx,\bft,\bfp_\bfx,\bfp_\bft\>$ on
$N_f := K[\bfx][\bff^{-1},\bfs]\bff^\bfs$ as follows: $x_i$ and $\p_{x_i}$,
for $i = 1, \ldots, n$, act naturally and
\begin{align*}
  t_j \cdot h(\bfx, s_1, \ldots, s_j, \ldots , s_r)  \bff^\bfs 
  &= h(\bfx, s_1, \ldots, s_j+1, \ldots , s_r) f_j \bff^\bfs,\\
  \p_{t_j} \cdot h(\bfx, s_1, \ldots, s_j,  \ldots , s_r)  \bff^\bfs 
  &= -s_j h(\bfx, s_1, \ldots, s_j-1, \ldots , s_r) f_j^{-1} \bff^\bfs,
\end{align*}
for $j = 1,\ldots,r$ and $h \in K[\bfx][\bff^{-1},\bfs]$.

With $\sigma = -\left( \sum_{i=1}^r \p_{t_i}t_i \right)$, the \emph{generalized Bernstein--Sato polynomial} $b_{\bff,g}$ of $\bff$ at $g\in K[\bfx]$ is the monic polynomial $b\in\bC[\bfs]$ of the lowest degree for which there exist $P_k \in \Dx\< \p_{t_i} t_j \mid 1\leq i,j\leq r \>$ for $k = 1,\ldots,r$ such that
\begin{align}\label{eq:generalB}
b(\sigma)g\bffs = \sum_{k=1}^r P_k g f_k \bffs.
\end{align}

\begin{remark}\label{rem: reason generalB}
When $r=1$, the generalized \BS\ polynomial $b_{\bff,1} = b_\bff$ is the global \BS\ polynomial of $\bff = f_1$ discussed in Section~\ref{sec:global}.
\end{remark}

There is again an equivalent definition of $b_{\bff,g}$ by way of the $V$-filtration.
To state this, let $V^\bullet \Dxt$ denote the $V$-filtration of $\Dxt$ along $X$, where $V^m\Dxt$ is $\Dx$-generated by the set $\{ \bft^\mu \bfp_\bft^\nu \mid |\mu| - |\nu| \geq m \}$.
The following statement may be taken as the definition of the $V$-filtration on $K[\bfx]$.
\begin{theorem}\cite[Theorem~1]{bms} \label{thm:1 bms}
For $c\in\bQ$ and sufficiently small $\epsilon>0$,
$\multI{\bff}{c} = V^{c+\epsilon}K[\bfx]$ and $V^cK[\bfx] = \multI{\bff}{c-\epsilon}$.
\end{theorem}

Consider the graph of $\bff$, the map
$i_\bff: X \rightarrow Y$ defined by $i_\bff(\bfx) = (\bfx,f_1(\bfx),\dots,f_r(\bfx))$.
We denote the $D$-module direct image of $K[\bfx]$ along $i_\bff$ by
\begin{align}\label{eq:Mf}
M_\bff := (i_\bff)_+ K[\bfx] \cong K[\bfx]\otimes_K K\<\bfp_\bft\>.
\end{align}
This module carries a $\Dxt$-module structure,
where the action of a vector field $\xi$ on $X$ and that of $t_j$ are given by
\begin{align*}
\xi(p\otimes \bfp_\bft^\nu) =
\xi p \otimes \bfp_\bft^\nu - \sum_{i=1}^r (\xi f_i) p\otimes \bfp_\bft^{\nu + e_j}
\quad \text{and} \quad
t_j\cdot (p\otimes \bfp_\bft^\nu) =
f_j p \otimes \bfp_\bft^\nu - \nu_j p \otimes \bfp_\bft^{\nu - e_j},
\end{align*}
where $\bfp_\bft^\nu = \prod_{i=1}^r \p_{t_i}^{\nu_i}$
for $\nu = (\nu_1,\dots, \nu_r) \in \bN$
and $e_j$ is  the element of $\bN^r$ with $j$-th component equal to $1$
and all others equal to $0$.

Further, $M_\bff$ admits a $V$-filtration with
\[
V^m M_\bff = \sum_{\nu\in\bN^r} (V^{m+|\nu|} K[\bfx]) \otimes \bfp_\bft^\nu.
\]

For a polynomial $g\in K[\bfx]$ so that $g\otimes 1\in M_\bff$,
$b_{\bff,g}$ is equal to the monic minimal polynomial of the action of $\sigma$ on
\[
\ol{M}_{\bff,g} := \frac{(V^0 \Dxt) (g\otimes 1)}{(V^1 \Dxt) (g\otimes 1)}.
\]

\begin{remark}\label{rem:embedding}
There is a canonical embedding of $M_\bff$ into $N_\bff$, where $s_i$ is identified with $-\p_{t_i}t_i$. In particular, for a natural number $m$, the image of $(V^m \Dxt) (1\otimes 1)$ under this embedding is contained in $(V^0\Dxt) \<\bff\>^m \bff^\bfs \subseteq N_\bff$.
\end{remark}

\subsection{Algorithms}
To compute the generalized \BS\ polynomial, we define the left $\Dxt$-ideal
\[
I_\bff = \< t_i-f_i \mid 1\leq i\leq r \> + \< \p_{x_j} + \textstyle\sum_{i=1}^r \textstyle\frac{\p f_i}{\p t_j} \p_{x_i} \mid 1\leq j\leq n \>
\]
that appears in the following multivariate analog of Theorem~\ref{thm:one var If}. Recall that $\sigma = -\left( \sum_{i=1}^r \p_{t_i}t_i \right)$.
\begin{theorem}
The ideal $I_\bff$ is equal to $\Ann_{\Dxt} \bff^\bfs$.
Furthermore, the ideal $\Ann_{\Dx [s]} \bff^\bfs$ equals the image of
$I_\bff \cap \Dx [\sigma]$ under the substitution $\sigma \mapsto s$.
\end{theorem}

We now provide two subroutines in our computations of \BS\ polynomials and multiplier ideals.
The first finds the left side of a functional equation of the form \eqref{eq:generalB} without an expensive elimination step.
The second finds the homogenization of a $\Dxt$-ideal with respect to the weight vector $(-w,w)$, where $w = (\bf{0},\bf{1})\in \rr^n\times \rr^r$ determines an elimination term order for $X$ in $Y$.

\begin{algorithm} \label{alg: LAtrick} $b = \LAtrick(g,G)$
\begin{algorithmic}\ifx\form{PREP} [1] \fi
\REQUIRE generators $G$ of an ideal $I \subset \Dxt$,\\ a polynomial $g\in K[\bfx]$,\\
such that there is $b\in K[s]$ with $b(\sigma)g\in I$.
\ENSURE $b$, the monic polynomial of minimal degree such that $b(\sigma)g \in I$.
\STATE $B \leftarrow \{ a \mbox{ Gr\"obner basis of } \Dxt G \}$.
\STATE $d \leftarrow 0$.
\REPEAT
\STATE $d \leftarrow d + 1$
\UNTIL{$\exists (c_0,\ldots,c_{d})\in K^{d+1}$ such that $c_d=1$ and
	\[
	\sum_{i=0}^d c_i \NF_B(\sigma^i g) = 0.
	\] }
\RETURN $\sum_{i=0}^d c_i s^i$.
\end{algorithmic}
\end{algorithm}

\begin{algorithm} \label{alg: star} $G^* = \starIdeal(G,w)$
\begin{algorithmic}\ifx\form{PREP} [1] \fi
\REQUIRE generators $G$ of an ideal $J\subset \Dxt$,\\ 
	a weight vector $w \in \bZ^{n+r}$.
\ENSURE $G^* \subset \gr_{(-w,w)}\Dxt \cong \Dxt$, a set of generators of the ideal $J^*$ of $(-w,w)$-homogeneous elements of $J$.
\STATE $G^h \leftarrow$ generators $G$ homogenized w.r.t. a weight $(-w,w)$; 
	$G^h \subset \Dxt[h]$ with a homogenizing variable $h$ of weight 1.
\STATE $B \leftarrow \{ \mbox{a Gr\"obner basis of } (G^h,hu-1)\subset \Dxt[h,u]$
	w.r.t. a monomial order eliminating $\{h,u\} \}$.
\RETURN $B \cap \Dxt$.
\end{algorithmic}
\end{algorithm}

Below are two algorithms that are simplified versions of Shibuta's algorithms
for the generalized \BS\ polynomial. In the first, we use a module $\Dxt[s]$,
where the new variable $s$ commutes with all variables in $\Dxt$.

\begin{algorithm}\label{alg:generalB star} 
$b_{\bff,g} = \generalB(\bff,g, \text{\tt StarIdeal} )$ 
\begin{algorithmic}\ifx\form{PREP} [1] \fi
\REQUIRE $\bff = \{f_1,\ldots,f_r\} \subset K[\bfx]$, $g\in K[\bfx]$.
\ENSURE $b_{\bff,g}$, the generalized \BS\ polynomial of $\bff$ at $g$.
\STATE 
\begin{tabbing}
$G_1 \leftarrow $ \= $\{t_j-f_j \mid j=1,\ldots,r\}$ $\cup$ $\{\p_{x_i}+\sum_{j=1}^r \frac{\p f_j}{\p x_i} \p_{t_j} \mid i=1,\ldots,n \}$.
\end{tabbing}
\STATE
\begin{tabbing}
$G_2 \leftarrow $ \= $\starIdeal(G_1,w)$ $\cup\ \langle g f_i \mid 1\leq i \leq r \rangle \cup \{s-\sigma\} \subset \Dxt[s]$,
where $w$ assigns weight $1$\\ to all $\p_{t_j}$ and $0$ to all $\p_{x_i}$.\end{tabbing}
\RETURN $\LAtrick(G_2)$.
\end{algorithmic}
\end{algorithm}

\begin{algorithm}\label{alg:generalB initial} $b_{\bff,g} = \generalB(\bff,g, \text{\tt InitialIdeal} )$
\begin{algorithmic}\ifx\form{PREP} [1] \fi
\REQUIRE $\bff = \{f_1,\ldots,f_r\} \subset K[\bfx]$, $g\in K[\bfx]$.
\ENSURE $b_{\bff,g}$, the generalized \BS\ polynomial of $\bff$ at $g$.
\STATE 
\begin{tabbing}
$G_1 \leftarrow $ \= $\{t_j-f_j \mid j=1,\ldots,r\}$ 
$\cup$ $\{\p_{x_i}+\sum_{j=1}^r \frac{\p f_j}{\p x_i} \p_{t_j} \mid i=1,\ldots,n \}$.
\end{tabbing}
\STATE $G_2 \leftarrow G_1 \cap \Dxt \cdot g$.
\STATE $G_3 \leftarrow \text{generators of }\IN_{(-w,w)}\<G_2\>$, 
        where $w$ assigns weight $1$ to all $\p_{t_j}$ and $0$ to all $\p_{x_i}$.
\RETURN $\LAtrick(G_3)$.
\end{algorithmic}
\end{algorithm}

Their correctness follows from \cite[Theorems 3.4 and 3.5]{shibuta}.
\begin{remark} According to the experiments in
\cite{Levandovskyy-Morales:comparison} a modification of
Algorithm~\ref{alg: star} that uses elimination
involving one less additional variable exhibits better performance.
Our current implementation does not take advantage of this. 
\end{remark}

\subsection{Applications}

The study of the generalized \BS\ polynomial in \cite{bms} yields several applications of our algorithms, which we mention here. Each has been implemented in Macaulay2.

We begin with a result that shows that comparison with the roots of $b_{\bff, g}(s)$ provides a membership test for $\multI{\bff}{c}$ for any positive rational number $c$.

\begin{proposition}\cite[Corollary~2]{bms} \label{cor:2 bms}
Let $g\in K[\bfx]$ and fix a positive rational number $c$.
Then $g\in\multI{\bff}{c}$ if and only if $c$ is strictly less than all roots of $b_{\bff, g}(-s)$.
\end{proposition}

When $\bff$ defines a complete intersection, 
Algorithms~\ref{alg:generalB star} and~\ref{alg:generalB initial}
provide tests to determine if $Z$ has at most rational singularities.

\begin{theorem}\cite[Theorem~4]{bms}
Suppose that $Z$ is a complete intersection of codimension $r$ in $Y$ defined by $\bff = f_1,\dots, f_r$.
Then $Z$ has at most rational singularities if and only if $\lct(\bff) = r$ and has multiplicity one as a root of $b_\bff(-s)$.
\end{theorem}

To compute a local version of the generalized \BS\ polynomial, we need the following analog of Theorem~\ref{thm: BSxs}.

\begin{theorem} \label{thm: multiBSxs}
Let $b(\bfx,s)$ be a nonzero polynomial in $K[\bfx, s]$. 
Then the polynomial $b(\bfx, \sigma) \in \IN_{(-w,w)} I_{\bff} \cap K[\bfx, \sigma]$ 
if and only if there exist $Q_k\in D[s]$ s.t. $\sum_{k=1}^r Q_k f_k \bff^s = b(\bfx,s)\bff^s$.
\end{theorem}
\begin{proof}
This follows by the same argument as that of Theorem~\ref{thm: BSxs}.
\end{proof}

\begin{remark}\label{rem:local generalB}
In light of Theorem~\ref{thm: multiBSxs}, the strategy in Section~\ref{sec:local} yields a computation of the local version of the generalized \BS\ polynomial.
The only significant difference comes from the lack of an analogue to the map $\psi$ of Remark~\ref{rem:homogenize}.
However, it is still possible to compute $\IN_{(-w,w)}I_\bff \cap K[\bfx, \sigma]$ by adjoining one more variable $s$ to the algebra and $s-\sigma$ to the ideal and eliminating $\bft$ and $\bfp_\bft$. In case of the hypersurface this is a more expensive strategy than the one described in Remark~\ref{rem:homogenize}.
\end{remark} 

\section{M\lowercase{ultiplier ideals via m-generalized} \BS\ 
\lowercase{polynomials}}
\label{sec:bsm+mult}

For this section, we retain the notation of Section~\ref{sec:generalized} and discuss Shibuta's $m$-generalized \BS\ polynomials. These are defined using the $V$-filtration of $\Dxt$ along $X$, but they also possess an equational definition.
In contrast to the generalized \BS\ polynomials of Section~\ref{sec:generalized}, this generalization allows us to simultaneously consider families of polynomials $K[\bfx]$, yielding a method to compute multiplier ideals.

\begin{definition}\label{def: generalBm}
Let $\overline{M}_\bff^{(m)} := (V^0\Dxt)\delta / (V^m\Dxt)\delta$ with $\delta = 1\otimes 1 \in M_\bff\cong K[\bfx]\otimes_K K\<\bfp_\bft\>$.
Define the \emph{$m$-generalized \BS\ polynomial}
$b_{f,g}^{(m)}$ to be the monic minimal polynomial of the action of
$\sigma := -\left( \sum_{i=1}^r \p_{t_i}t_i \right)$ on
\[
\ol{M}_{\bff,g}^{(m)} :=
(V^0 \Dxt) \overline{g\otimes 1} \subseteq \overline{M}_\bff^{(m)}.
\]
\end{definition}

\begin{remark}
Since $M_\bff$ is $V$-filtered, the polynomial $b_{f,g}^{(m)}$ is nonzero and its roots are rational.
\end{remark}

\begin{proposition}
\label{prop: bm eq}
The $m$-generalized \BS\ polynomial
$b_{\bff,g}^{(m)}$ is equal to
the monic polynomial $b(s)$ of minimal degree in $K[s]$ such that there exist
$P_k\in  \Dx \<-\p_{t_i} t_j \mid 1\leq i,j \leq r \>$ and $h_k\in \<\bff\>^m$ such that in $N_\bff$ there is an equality
\begin{align}\label{eq:equation for bm}
b(\sigma) g \bff^\bfs = \sum_{k=1}^r P_k h_k \bff^\bfs.
\end{align}
\end{proposition}
\begin{proof}
By the embedding in Remark~\ref{rem:embedding}, the existence of such an equation is equivalent to the existence of
$Q_k\in \Dx \<-\p_{t_i} t_j \mid 1\leq i,j \leq r \>$ and $\mu(k) \in \bN^r$ with $|\mu(k)| \geq m$ such that in $M_\bff$,
$b(\sigma) \cdot (g\otimes 1) = \sum_{k=1}^r Q_k \bft^{\mu(k)} \cdot (1\otimes 1)$.
\end{proof}

\begin{remark}
Since $(V^0\Dxt)\ol{g\otimes 1}\subseteq \overline{M}_\bff^{(1)}$ is a quotient of $\ol{M}_{\bff,g}$, the  generalized \BS\ polynomial $b_{\bff,g}$ is a multiple of the $m$-generalized \BS\ polynomial $b_{\bff,g}^{(1)}$.
When $g$ is a unit, the equality $b_{\bff,g} = b_{f,g}^{(1)}$ holds, as seen easily by comparing \eqref{eq:generalB} and \eqref{eq:equation for bm}.
However, this equality does not hold in general.
\end{remark}

\begin{example}\label{ex:different generalB1}
When $n=3$ and $\bff = \sum_{i=1}^3 x_i^2$, we have
\begin{align*}
b_{\bff,x_1}(s) = (s+1)(s+\frac{5}{2})
\quad \mbox{and}\quad b_{\bff,x_1}^{(1)}(s) = s+1.
\end{align*}
In particular, $b_{\bff,x_1}^{(1)}$ strictly divides $b_{\bff,x_1}$.
\end{example}

Proposition~\ref{prop: bm eq} translates into the following algorithm.
\begin{algorithm} \label{alg: generalB} $b_{\bff,g}^{(m)} = \generalB(\bff,g,m)$
\begin{algorithmic}\ifx\form{PREP} [1] \fi
\REQUIRE $\bff = \{f_1,\ldots,f_r\} \subset K[\bfx]$, $g\in K[\bfx]$, $m\in\bZ_{>0}$.
\ENSURE $b_{\bff,g}^{(m)}$, the $m$-generalized \BS\ polynomial (as defined in Definition~\ref{def: generalBm}).
\STATE 
\begin{tabbing}
$G_1 \leftarrow $ \= $\{t_j-f_j \mid j=1,\ldots,r\}$ $\cup$ $\{\p_{x_i}+\sum_{j=1}^r \frac{\p f_j}{\p x_i} \p_{t_j} \mid i=1,\ldots,n \}$.
\end{tabbing}
\STATE $G_2 \leftarrow \starIdeal(G_1,w) \cup \{ \bff^\alpha \mid \alpha\in\bN^r,\,|\alpha| = m \} $, 
        where $w$ assigns weight $1$ to all $\p_{t_j}$ and $0$ to all $\p_{x_i}$.
\RETURN $\LAtrick(g,G_2)$.
\end{algorithmic}
\end{algorithm}

\subsection{Jumping coefficients and the log canonical threshold}
\label{subsec:lct}
For the remainder of this article, set $K=\bC$.
Our algorithms for multiplier ideals are motivated by the following result.

\begin{theorem}\cite[Theorem~4.3]{shibuta} \label{thm: shibuta 4.2}
For $g\in K[\bfx]$ and $c < m + \lct(\bff)$,
$g\in\multI{\bff}{c}$ if and only if
$c$ is strictly less than every root of $b_{\bff,g}^{(m)}(-s)$. In other words,
\[
\multI{\bff}{c} = \{ g\in K[\bfx] \mid b_{\bff,g}^{(m)} (-\alpha) = 0 \Rightarrow c < \alpha \}.
\]
\end{theorem}

\begin{proof}
By Theorem~\ref{thm:1 bms}, 
$\multI{\bff}{c} = V^{c+\epsilon}K[\bfx]$ for all sufficiently small $\epsilon>0$.
Hence, $g\in\multI{\bff}{c}$ precisely when
$g\otimes 1 \in V^\alpha M_\bff$ for all $\alpha\leq c$,
or equivalently,
\begin{align}\label{eq: key ineq}
c < \max \{ \alpha \mid g\otimes 1 \in V^\alpha M_\bff \}.
\end{align}
As in \cite[(2.3.1)]{bms},
the right side of \eqref{eq: key ineq} is equal to
$\min \{ \alpha \mid \Gr_V^\alpha ((V^0 \Dxt)(g\otimes 1)) \neq 0 \}$
and strictly less than
$\min \{ \alpha \mid \Gr_V^\alpha ((V^m \Dxt)\delta) \neq 0 \}$.
Thus by our choice of $m$,
$g\in\multI{\bff}{c}$ exactly when
$c$ is strictly greater than
$\min \{\alpha \mid \Gr_V^\alpha \ol{M}_{\bff,g}^{(m)} \neq 0 \}$.
The theorem now follows because
$\Gr_V^\alpha \ol{M}_{\bff,g}^{(m)} \neq 0$ if and only if
$b_{\bff,g}^{(m)}(-\alpha) = 0$.
\end{proof}

Theorem~\ref{thm: shibuta 4.2} provides a second membership test for membership in $\multI{\bff}{c}$;
moreover, the following corollary provides a method for computing the log canonical threshold and jumping coefficients of $\bff$ via the $m$-generalized \BS\ polynomial $b_{\bff}^{(1)} =b_{\bff,1}^{(1)}$.

\begin{cor}\label{cor:jumping coefficients}
For any positive integer $m$, the minimal root of $b_\bff^{(m)}(-s)$ is equal to the log-canonical threshold $\lct(\bff)$ of $\<\bff\> \subseteq K[\bfx]$. Further, the jumping coefficients of $\<\bff\>$ within the interval $[\lct(\bff), \lct(\bff) + m)$ are all roots of $b_\bff^{(m)}(-s)$.
\end{cor}

\subsection{Computing multiplier ideals}
\label{subsec:multiplier}

Here we present an algorithm to compute multiplier ideals that simplifies the method of Shibuta \cite{shibuta}. In particular, significant improvement is achieved bypassing the primary decomposition computations required by Shibuta's method.

For a positive integer $m$, define the $K[\bfx,\sigma]$-ideal
\[
J_\bff(m) = \left( I_\bff^* + \Dxt\cdot \<\bff\>^m \right) \cap K[\bfx,\sigma],
\]
where $I_\bff^* \subset\Dxt$ is the ideal of the $(-w,w)$-homogeneous elements of $I_\bff$.
This ideal is closely related to the $m$-generalized \BS\ polynomials.

\begin{lemma}\label{lem:bm def J}
For $g\in K[\bfx]$,
the $m$-generalized \BS\ polynomial $b_{\bff,g}^{(m)}$ is equal to the monic polynomial $b(s)\in K[s]$ of minimal degree such that
\begin{align}\label{eq: bm and J}
\< b(\sigma) \> = (J_\bff(m) :g) \cap K[\sigma].
\end{align}
\end{lemma}
\begin{proof}
By \eqref{eq:equation for bm}, $b_{\bff,g}^{(m)}$ is the monic polynomial $b(s)\in K[s]$ of minimal degree such that
\[
b(\sigma)g \in I_\bff + \Dx\< -\p_{t_i} t_j \mid 1\leq i,j \leq r \> \cdot \< \bff \>^m.
\]
Since $b(\sigma)g$ is $(-w,w)$-homogeneous, we obtain \eqref{eq: bm and J}.
\end{proof}

\begin{theorem}\cite[Theorem~4.4]{shibuta} \label{thm: shibuta 4.3}
Let $J_\bff(m) = \bigcap_{i=1}^l q_i$ be a primary decomposition with
$q_i\cap K[\sigma] = (\sigma + c(i) )^{\kappa(i)}$ for some positive integer $\kappa(i)$.
Then for $c< \lct(\bff) + m$,
\[
\multI{\bff}{c} = \bigcap_{j:\hspace{.3ex} c(j)\geq c} (q_j \cap K[\bfx]).
\]
\end{theorem}
\begin{proof}
We see from \eqref{eq:equation for bm} that $b_{\bff,g}^{(m)}(s)$ is the monic polynomial $b(s)$ of minimal degree such that there exist some $P_k\in  \Dx \<-\p_{t_i} t_j \mid 1\leq i,j \leq r \>$ and $h_k\in \<\bff\>^m$ such that
$(b(\sigma)g - \sum_{k} P_k h_k) \in I_\bff$.
Equivalently,
\[
b(\sigma)g \in \left(I_\bff^* +  \Dxt \cdot \<\bff\>^m \right) \cap K[\bfx,\sigma].
\]
The theorem now follows from Lemma~\ref{lem:bm def J}.
\end{proof}

The following is based on methodology used in the computation of the local $b$-function and, in particular, employs Algorithm~\ref{alg: exceptional locus core}. Its correctness follows immediately from Theorem~\ref{thm: shibuta 4.3} and the results of Section~\ref{sec:local}.

\begin{algorithm} \label{alg: multI} $ \multI{\bff}{c} = \multiplierIdeal(\bff,c)$
\begin{algorithmic}\ifx\form{PREP} [1] \fi
\REQUIRE $\bff = \{f_1,\ldots,f_r\} \subset K[\bfx]$, $c\in\bQ$.
\ENSURE $\multI{\bff}{c}$, the multiplier ideal of $\bff$ with coefficient $c$.
\STATE 
\begin{tabbing}
$G_1 \leftarrow $ \= $\{t_j-f_j \mid j=1,\ldots,r\}$ $\cup$ $\{\p_{x_i}+\sum_{j=1}^r \frac{\p f_j}{\p x_i} \p_{t_j} \mid i=1,\ldots,n \}$.
\end{tabbing}
\STATE $m \leftarrow \lceil \max\{c-\lct(\bff), 1\} \rceil$.
\IF{$c-\lct(\bff)$ is integer and $\geq 1$}
\STATE $m \leftarrow m + 1$
\ENDIF
\STATE \begin{tabbing}$G_2 \leftarrow \starIdeal(G_1,w)$ \= $\cup\ \{ \bff^\alpha \mid \alpha\in\bN^r,\,|\alpha| = m \}$
 $\cup\ \{s-\sigma\} \subset \Dxt[s]$,
       where $w$ assigns\\ weight $1$ to all $\p_{t_j}$ and $0$ to all $\p_{x_i}$.
       \end{tabbing}
\STATE $B \leftarrow \left\{ \ 
\mbox{a Gr\"obner basis of } G_2 \mbox{ w.r.t. an order eliminating } \{\p_\bfx, \bft, \p_\bft\}\ 
\right\} \,\cap\, K[\bfx,s]$. \label{alg: multI elimination step}
\STATE \label{alg: multI generalB} $b \leftarrow  \generalB(\bff,1,m)$. 
\COMMENT{The computation of $b_{\bff,1}^{(m)}$ may make use of $B$.}
\STATE{$b' \leftarrow $ product of factors $(s-c')^{\alpha(c')}$ of $b$ over all roots $c'$ of $b_f^{(m)}$ such that  $-c'>c$, where $\alpha(c')$ equals the multiplicity of the root $c'$.} \label{alg: multI b'}
\RETURN $\locusCore(B,b')$.
\end{algorithmic}
\end{algorithm}

As noted in \cite[Remark~4.6.ii]{shibuta}, 
$\multI{\bff}{c} = \<\bff\>\cdot \multI{\bff}{c-1}$ when $c$ is at least equal to the analytic spread $\lambda(\bff)$ of $\<\bff\>$.
(The analytic spread of $\<\bff\>$ is the least number of generators for an ideal $I$ such that $\<\bff\>$ is integral over $I$.)
Hence, to find generators for any multiplier ideal of $\bff$,
it is enough to compute $J_\bff(m)$ for one $m \geq \lambda(\bff) - \lct(\bff)$.

When it is known that the multiplier ideal $\multI{\bff}{c}$ is 0-dimensional, it is possible to bypass the elimination step (line \ref{alg: multI elimination step} of Algorithm~\ref{alg: multI}) in the following fashion.
For a fixed monomial ordering $\geq$ on $K[\bfx]$,
we know that there are finitely many standard monomials
(monomials not in the initial ideal $\IN_\geq \multI{\bff}{c})$.
Let $b'\in\bQ[s]$ be the polynomial produced by lines \ref{alg: multI generalB}
and \ref{alg: multI b'} of the above algorithm.
A basis for the $K$-linear relations amongst
$\{ \bfx^\alpha b'(\sigma) \mid |\alpha|\leq d \}$ modulo $J_\bff(m)$ gives
a basis $P_d$ for the $K$-space of polynomials in $\multI{\bff}{c}$ up to degree $d$.
By starting with $d=0$ and incrementing $d$
until all monomials of degree $d$ belong to $\IN_\geq\<P_{d-1}\>$,
we obtain $\<P_d\> = \multI{\bff}{c}$ upon termination.

\begin{algorithm}\label{alg: multI-lin-alg} $\multI{\bff}{c}=\multiplierIdealLinAlg(\bff,c,d_\text{max})$
\begin{algorithmic}\ifx\form{PREP} [1] \fi
\REQUIRE $\bff = \{f_1,\ldots,f_r\} \subset K[\bfx]$, $c\in\bQ$, $d_\text{max}\in\bN$.
\ENSURE the multiplier ideal $\multI{\bff}{c}\subset K[\bfx]$, when it is generated in degrees at most $d_\text{max}$.
\STATE 
\begin{tabbing}
$G_1 \leftarrow $ \= $\{t_j-f_j \mid j=1,\ldots,r\}$ $\cup$ $\{\p_{x_i}+\sum_{j=1}^r \frac{\p f_j}{\p x_i} \p_{t_j} \mid i=1,\ldots,n \}$.
\end{tabbing}
\STATE $m \leftarrow \lceil \max\{c-\lct(\bff), 1\} \rceil$.
\IF{$c-\lct(\bff)$ is an integer and $\geq 1$}
\STATE $m \leftarrow m + 1$
\ENDIF
\STATE \begin{tabbing}$G_2 \leftarrow \starIdeal(G_1,w)$ \= $\cup\ \{ \bff^\alpha \mid \alpha\in\bN^r,\,|\alpha| = m \}  \subset \Dxt$,
where $w$ assigns weight $1$ to\\ all $\p_{t_j}$ and $0$ to all $\p_{x_i}$.        \end{tabbing}
\STATE $B \leftarrow \left\{
\begin{array}{c}
\mbox{a Gr\"obner basis of } G_2 \mbox{ w.r.t. }
\mbox{{\bf any monomial order}}
\end{array}
\right\}$.
\STATE $b \leftarrow  \generalB(\bff,1,m)$.
\STATE{$b' \leftarrow $ product of factors $(s-c')^{\alpha(c')}$ of $b$ over all roots $c'$ of $b_f^{(m)}$ such that  $-c'>c$, where $\alpha(c')$ equals the multiplicity ot the root $c'$.}
\STATE $d \leftarrow -1$; $P \leftarrow \emptyset \subset K[\bfx]$ with $\geq$ that respects degree.
\WHILE{$P = \emptyset$ \textbf{or} $(\IN_\geq\<P\>$ does not contain all monomials of degree $d$ \textbf{and} $d<d_\text{max})$}
\STATE $d \leftarrow d + 1$,
\STATE $A \leftarrow \{\alpha \mid |\alpha|\leq d, \, \bfx^\alpha \notin \IN_\geq\<P\>\}$.
\STATE Find a basis $Q$ for the $K$-syzygies $(q_\alpha)_{\alpha\in A}$ such that
	\[
	\sum_{\alpha\in A} q_\alpha \NF_B(\bfx^\alpha b(\sigma)) = 0.
	\]
\STATE $P \leftarrow P \cup \{ \sum_{\alpha\in A} q_\alpha \bfx^\alpha \mid (q_\alpha) \in Q\}$.
\ENDWHILE
\RETURN $\<P\>$.
\end{algorithmic}
\end{algorithm}

Notice that with $d_{max}=\infty$ the algorithm terminates in case $\dim \multI{\bff}{c} = 0$.  It also can be used to provide a $K$-basis of the up-to-degree-$d_\text{max}$ part of an ideal of any dimension.

\section{E\lowercase{xamples}}
\label{sec:examples}

We have tested our implementation on the problems in \cite{shibuta}. In addition, this section provides examples from other sources with the theoretically known \BS\ polynomials, log-canonical thresholds, jumping numbers, and/or multiplier ideals; below is the output of our algorithms on several of them.

The authors would like to thank Zach Teitler for suggesting interesting examples, some of which are beyond the reach of our current implementation.
We also thank Takafumi Shibuta for sharing his script (written in risa/asir~\cite{risa-asir-www}), which is the only other existing software for computing multiplier ideals.

A note on how to access Macaulay2 scripts generating examples, including the ones in this paper and some unsolved challenges, is posted at \cite{Berkesch-Leykin:multiplier-ideals-www} along with other useful links.

\begin{example}\label{ex:nonJC root}
When $f = x^5 + y^4 + x^3y^2$, Saito observed that not all roots of $b_f(-s)$ are jumping coefficients \cite[Example~4.10]{saito on}.
The roots of $b_f(-s)$ within the interval $(0,1]$ are
\[
\frac{9}{20}, \frac{11}{20}, \frac{13}{20}, \frac{7}{10}, \frac{17}{20}, \frac{9}{10}, \frac{19}{20}, 1.
\] 
However, $\frac{11}{20}$ is not a jumping coefficient of $f$.
This can be seen in 
$J_f(1)$ from Theorem~\ref{thm: shibuta 4.3}, which has, among others, the primary components $\<s+\frac{9}{20},y,x\>$ and $\<s+\frac{11}{20},y,x\>$.
In fact,
\[
\multI{f}{c} = \begin{cases}
\cc[x,y] & \mbox{if $0 \leq c < \frac{9}{20}$,}\\
\< x,y \> & \mbox{if $\frac{9}{20} \leq c < \frac{13}{20}$,}\\
\< x^2,y \> & \mbox{if $\frac{13}{20} \leq c < \frac{7}{10}$,}\\
\< x^2,xy,y^2 \> & \mbox{if $\frac{7}{10} \leq c < \frac{17}{20}$,}\\
\< x^3,xy,y^2 \> & \mbox{if $\frac{17}{20} \leq c < \frac{9}{10}$,}\\
\< x^3,x^2y,y^2 \> & \mbox{if $\frac{9}{10} \leq c < \frac{19}{20}$,}\\
\< x^3,x^2y,xy^2,y^3 \> & \mbox{if $\frac{19}{20} \leq c < 1$,}
\end{cases}
\]
and $\multI{\bff}{c} = \<\bff\>\cdot \multI{\bff}{c-1}$ for all $c\geq1$.
%
%
\end{example}

\begin{example}
We compute \BS\ polynomials to verify examples corresponding to \cite[Example~7.1]{zach line arr}.
The $\cc[x,y,z]$-ideal
\[
\<\bff\> = \< x-z,y-z \> \cap \< 3x-z,y-2z \> \cap \< 5y-x,z \>
\]
defining three non-collinear points in $\bP^2$ has
\[
b_\bff(s) = (s+\frac{3}{2})(s+2)^2.
\]
In particular, its log canonical threshold is $\frac{3}{2}$.
The multiplier ideals in this case are
\[
\multI{f}{c} = \begin{cases}
\cc[x,y,z] & \mbox{if $0\leq c< \frac{3}{2}$,} \\
\< x,y,z \>  & \mbox{if $\frac{3}{2}\leq c<2$,}
\end{cases}
\]
and  $\multI{\bff}{c} = \<\bff\>\cdot \multI{\bff}{c-1}$ for all $c\geq2$.
On the other hand, the $\cc[x,y,z]$-ideal
\[
\<{\bf g}\> = \< y,z \> \cap \< x-2z,y-z \> \cap \< 2x-3z,y-z \>
\]
defines three collinear points in $\bP^2$.
Since
\[
b_{\bf g}(s) = (s+\frac{5}{3})(s+2)^2(s+\frac{7}{3}),
\]
the log canonical threshold of ${\bf g}$ is $\frac{5}{3}$.
Here the multiplier ideals are
\[
\multI{\bf g}{c} = \begin{cases}
\cc[x,y,z] & \mbox{if $0\leq c< \frac{5}{3}$,} \\
\< x,y,z \> & \mbox{if $\frac{5}{3} \leq c < 2$,} \\
\end{cases}
\]
and $\multI{\bf g}{c} = \<\bf g\>\cdot \multI{\bf g}{c-1}$ for all $c\geq2$.
Thus, as Teitler points out, although ${\bf g}$ defines a more special set than $\bff$, it yields a less singular variety.
\end{example}

\begin{example}
Consider $f = (x^2 - y^2)(x^2 - z^2)(y^2 - z^2)z$,
the defining equation for a nongeneric hyperplane arrangement.
Saito showed that $\frac{5}{7}$ is a root of $b_f(-s)$ but not a jumping coefficient \cite[5.5]{saito mult}.
We verified this, obtaining the root $1$ of $b_f(-s)$ with multiplicity 3, as well the following roots of multiplicity 1 (including $\frac{5}{7}$):
\[
\frac{3}{7}, \frac{4}{7}, \frac{2}{3}, \frac{5}{7}, \frac{6}{7}, \frac{8}{7}, \frac{9}{7}, \frac{4}{3}, \frac{10}{7}, \frac{11}{7}.
\]
Further,
\[
\multI{f}{c} = \begin{cases}
\cc[x,y,z] & \mbox{if $0 \leq c < \frac{3}{7}$,}\\
\< x,y,z \> & \mbox{if $\frac{3}{7} \leq c < \frac{4}{7}$,}\\
\< x,y,z \>^2& \mbox{if $\frac{4}{7} \leq c < \frac{2}{3}$,}\\
\<z,x\> \cap \<z,y\> \cap \\
\<y+z,x+z\> \cap \<y+z,x-z\> \cap\\
\<y-z,x+z\> \cap \<y-z,x-z\> & \mbox{if $\frac{2}{3} \leq c < \frac{6}{7}$,}\\
\<z,x\> \cap \<z,y\> \cap \\
\<y+z,x+z\> \cap \<y+z,x-z\> \cap\\
\<y-z,x+z\> \cap \<y-z,x-z\> \cap \\
\<z^{3},y z^{2},x z^{2},x y z,y^{3},x^{3},x^{2} y^{2}\> & \mbox{if $\frac{6}{7} \leq c < 1$,}
\end{cases}
\]
and $\multI{f}{c} = \<\bff\> \cdot \multI{\bff}{c-1}$ for all $c\geq1$.
%
\end{example}

All examples in this section involve multiplier ideals of low dimension. In our experience, Algorithm~\ref{alg: multI-lin-alg} for a multiplier ideal of positive yet low dimension with a large value for $d_\text{max}$ runs significantly faster than Algorithm~\ref{alg: multI}.
This is due to the avoidance of an expensive elimination step.

\bibliographystyle{plain}
\bigskip


\end{document}